\documentclass[12pt]{amsart}

\usepackage{etex} % for too many package error

\usepackage[margin=3cm]{geometry}
\usepackage{amscd}
\usepackage{amsmath}
\usepackage{amssymb}
\usepackage{amsthm}
\usepackage{fancyhdr}
\usepackage{verbatim}

\usepackage[usenames, dvipsnames]{color}
\usepackage{tikz}
\usepackage[all]{xy}
\usepackage[colorlinks=true, linkcolor=blue, citecolor=blue, pagebackref=true]{hyperref}
\usepackage[pagebackref=true]{hyperref}
\usepackage{mathrsfs}
\renewcommand{\mathcal}{\mathscr}
\renewcommand{\epsilon}{\varepsilon}
\renewcommand{\phi}{\varphi}
\newcommand{\cal}{\mathcal}

\newcommand{\dimH}{\textrm{dim}_{\mathcal H}}
\newcommand{\dimP}{\textrm{dim}_{\mathcal P}}

\numberwithin{equation}{section}

\theoremstyle{plain}
\newtheorem{theorem}{Theorem}[section]

\newtheorem{lemma}[theorem]{Lemma}

\newtheorem{proposition}[theorem]{Proposition}
\newtheorem{corollary}[theorem]{Corollary}

\theoremstyle{definition}
\newtheorem{remark}[theorem]{Remark}
\newtheorem{example}[theorem]{Example}

\newcommand{\R}{\mathbb R}

\DeclareMathOperator{\diam}{diam}

\title{\textbf{Mass transference principle: from balls to arbitrary shapes}}
\author{Henna Koivusalo}
\address{Henna Koivusalo\\ University of Vienna\\ Oskar Morgensternplatz 1\\ 1090 Vienna, Austria}\email{henna.koivusalo@univie.ac.at}
\author{Micha\l\ Rams}
\address{Micha\l\ Rams\\Institute of Mathematics\\ Polish
Academy of Sciences\\ ul.
\'Sniadeckich 8, 00-656 Warszawa, Poland }\email{rams@impan.pl}
\date{\today}
\begin{document}

%===============================================

\thispagestyle{empty}

\begin{abstract}
The mass transference principle, proved by Beresnevich and Velani in 2006, is a strong result that gives lower bounds for the Hausdorff dimension of limsup sets of balls. We present a version for limsup sets of open sets of arbitrary shape. 
\end{abstract}
\thanks{This project was supported by OeAD grant number PL03/2017. M.R. was supported by National Science Centre grant 2014/13/B/ST1/01033 (Poland).}
\maketitle

\section{Introduction}

For $(A_i)$ a sequence of subsets of $\R^d$, define the limsup set
\[
\limsup A_i =\bigcap _{n=1}^\infty\bigcup_{i\ge n} A_i.
\]
The geometry of limsup sets is of great importance in dimension theory, as large classes of fractal sets, including attractors of iterated function systems and random covering sets, are limsup sets. See \cite{AT} and \cite{FJJS} for discussion and references. Our main focus is on the following fundamental result on dimensions of limsup sets, from a 2006 article of Beresnevitch and Velani \cite{BV}. Let $\lambda$ denote the Lebesgue measure in $\R^d$.
\begin{theorem}[Mass transference principle]
Let $(B_i)$ be a family of balls in $[0,1]^d$ such that $\lambda(\limsup B_i)=1$. Let $a>1$ and for each $i$, let $E_i$ be a ball with the same center as $B_i$ but of diameter $(\diam B_i)^a$. Then
\[
\dimH \limsup E_i \geq \frac da.
\]
\end{theorem}
We note that this is only the part of the original statement which is related to the Hausdorff dimension of the limsup set; the Beresnevitch and Velani's paper also provides information on its Hausdorff measure. The earliest appearance of a dimension statement of this kind is \cite[Theorem 2]{J}. 

The mass transference principle has found a great many applications in calculating the Hausdorff dimension of limsup sets, in metric number theory and fractal geometry. It has also been generalized in several directions. For a recent development see \cite{AB}, where versions of this result with different, more general assumptions on $B_i$ and $\limsup B_i$ are established. 

Of particular interest for us is the generalisation of Wang, Wu, and Xu \cite{WWX}. In their work, under the assumption $\lambda(\limsup B_i)=1$, the authors let the sets $E_i$ to be ellipsoids with semiaxes $(\diam B_i)^{a_j}, 1\leq a_1\leq\ldots\leq a_d$ instead of balls of diameter $(\diam B_i)^a$. They obtain the lower bound 
\begin{equation}\label{eq:WWX}
\min_{1\le j\le d}\left\{ \frac{d+ja_j-\sum_{i=1}^ja_i}{a_j} \right\}
\end{equation}
for the Hausdorff dimension of $\limsup E_i$. In \cite[Section 6]{WWX} they also briefly address the related problem of relaxing the condition on the shapes of the sets $B_i$. We point out that as in \cite{BV}, also in \cite{WWX} a Hausdorff measure statement was proved, but it is the Hausdorff dimension statement that is relevant for our purposes.

%In the current work, we provide the following new interpretation of this formula: there exists $s\leq d$ depending only on $a_1,\ldots,a_d$ such that $\varphi^s(E_i)\geq\lambda(B_i)$, and this $s$ is the lower bound for $\dimH \limsup E_i$. Here $\varphi^s$ is what in the dimension theory of iterated function systems is known as Falconer's singular value function, see \cite{F2} and Section \ref{sec:singular value function}. 

In this note we will generalize this result to arbitrary shapes $E_i$: we will only assume that $E_i\subset B_i$ and that they are open and nonempty. We will provide a Hausdorff dimension bound for these sets, and also calculate their packing dimension directly. (The packing dimension claim also follows from the fact that $\limsup E_i$ is a dense $G_\delta$-set, see \cite{SV}.) The argument involves a generalization of what in the dimension theory of iterated function systems is known as Falconer's singular value function, see \cite{F2} and Section \ref{sec:singular value function}. 

The paper is organized as follows. In Section \ref{sec:singular value function} we introduce the generalized singular value function and discuss its properties, in particular its relation to the Hausdorff content. Except for Lemma \ref{lem:phis}, this section is not necessary for the proof of our main result Theorem \ref{thm:main}, but it explains why Corollary \ref{cor:cusps} follows. Our results are formulated in Section \ref{sec:theorems}, where we also present an example to show why the full Lebesgue measure assumption is necessary (in a sense, see \cite{AB}). The results are proved in Sections \ref{sec:construction}-\ref{sec:localdim}. 

\subsection*{Acknowledgements}

We thank the referees for many helpful comments, which helped to significantly improve the article.

\section{Singular value function}\label{sec:singular value function}

%\subsection{Definitions}

In 1988 Kenneth Falconer \cite{F2} introduced a function, the {\it singular value function}, which for an ellipsoid $E\subset\R^d$ with semiaxes $\alpha_1\geq\ldots\geq \alpha_d$ and parameter $s\in [0,d]$ assigns the value
\[
\varphi^s(E)=\alpha_1\alpha_2\ldots\alpha_m\alpha_{m+1}^{s-m},
\]
where $m=\lfloor s \rfloor$ is the largest integer not larger than $s$.  This notion was of crucial importance in calculating in \cite{F2} the dimension of certain self-affine sets. 

Observe that the singular value function is also implicit in the dimension bound of \cite{WWX}. Namely, up to a multiplicative constant depending only on $d$, $\varphi^{s_0}(E_i)=\diam(B_i)^d$, where $s_0$ is given by (1.1). This  is not a coincidence, as the singular value function played important role in \cite{WWX}. Up to a multiplicative constant, the singular value function agrees with the better known quantity of {\it Hausdorff content}
\[
\cal H^s_\infty(A)=\inf\{\sum_{i=1}^\infty (\diam D_i)^s \mid A\subset \cup_{i=1}^\infty D_i\}, 
\]
where the sets $D_i$ are, say, closed balls. In order to find a mass transference principle for general shapes, we look for a version of the singular value function that applies to all Borel sets and is also equivalent to the Hausdorff content. We come to the following formula for a Borel set $E\subset\R^d$ 
\begin{equation}\label{eq:phi}
\phi^s(E)=\sup_\mu \inf_{x\in E}\inf_{r>0}\frac{r^s}{\mu(B_r(x))},
\end{equation}
where the supremum is taken over Borel probability measures supported on $E$ and $B_r(x)$ denotes the ball of radius $r$ and center $x$. It is not hard to see that when $E$ is an ellipsoid this formula returns Falconer's singular value function (up to a multiplicative constant), so that our definition is indeed a generalization. 

We note here that the function $\phi^s$ is not the only way of approximating the Hausdorff content. While studying large intersection properties of some random covering sets, Persson \cite{P} defined an approximating function $g_s$, which was later related to the Hausdorff content under a positive density assumption, see \cite[Lemmas 3.2, 3.9]{FJJS}. Persson's definition is applicable for sets of positive Lebesgue measure. 

The following proposition relates the singular value function and the Hausdorff content for all bounded Borel sets $E$. We have formulated it for a bounded Borel set in a Euclidean space, but note that the proof for Suslin sets (analytic sets) \cite{C} in compact metric spaces is the same. The second inequality is actually Frostman's Lemma \cite[Theorem 8.8]{M}, only with better constant: by a simple bootstrapping argument we show that if Frostman Lemma holds, it holds with constant $6^s$.

\begin{proposition}\label{prop:content}
Let $E$ be a bounded Borel set. Then 
\[
\varphi^s(E)\le \cal H^s_\infty (E) \le 6^s\varphi^s(E). 
\]
\end{proposition}
\begin{proof}
Let $\epsilon>0$. Find a probability measure $\mu$ such that for every $r>0$ and every $x\in E$ we have
\[
\mu(B_r(x)) < \frac{r^s}{\varphi^s(E)-\varepsilon}.
\]
As every ball intersecting $E$ is contained in a ball centred in $E$ of twice the radius, without assuming $x\in E$ we still have
\[
\mu(B_r(x)) < \frac{2^s r^s}{\varphi^s(E)-\varepsilon} = \frac{(\diam B_r(x))^s}{\varphi^s(E)-\varepsilon}.
\]
Thus, for any collection of balls $D_i$ covering $E$ we have
\[
\sum_i (\diam D_i)^s > (\varphi^s(E)-\varepsilon) \sum_i \mu(D_i) \geq \varphi^s(E)-\varepsilon.
\]
This proves the first inequality. 

We will need some preparation to prove the second inequality. First, for $\eta>0$, let $\mu_\eta$ be a Borel probability measure supported on $E$ such that

\begin{equation}\label{eq:choicemu}
\inf_{x\in E} \inf_{r>0} \frac {r^s} {\mu_\eta(B_r(x))} > \varphi^s(E) (1- \eta), 
\end{equation}
and denote 
\[
Z=\inf_{x\in E} \inf_{r>0} \frac {r^s} {\mu_\eta(B_r(x))}\le \phi^s(E). 
\]
For $\varepsilon>0$ let $A=A_{\eta, \varepsilon}$ be the set of those points $x\in E$ for which the following is true: for all $y\in E$ and $r>0$, if $x\in B_r(y)$ then $\mu_\eta(B_r(y)) < (1-\varepsilon)r^s/\varphi^s(E)$. 

By Frostman's Lemma (\cite[Section II]{C}, \cite[Theorem 8.8]{M}), there exists  a constant $c_d$ only depending on $d$ such that whenever $\cal H^s(K)>0$ for a bounded Borel set $K$, there is a Borel probability measure $m$ supported on $K$ such that

\[
m(B_r(x)) \leq \frac {1} {c_d{\cal H}_\infty^s(K)} r^s
\]
for all $r>0, x\in \R^d$, where $c_d$ only depends on $d$.

We wish to prove the following fact; that $\cal H^s_\infty (A)$ is relatively small. 

\begin{lemma} \label{lem:Hslittle}
\[
{\cal H}^s_\infty (A)\leq \frac{\varphi^s(E) }{ c_d(1-\varepsilon+\tfrac \varepsilon \eta)}.
\]
\end{lemma}

\begin{proof}[Proof of Lemma \ref{lem:Hslittle}]
Assume  without loss of generality that ${\cal H}^s_\infty(A)>0$. Then also $\cal H^s(A)>0$ and by \cite[Section II]{C}, \cite[Theorem 8.8]{M}, as above, there exists  a measure $\nu$ supported on $A$ such that

\[
\nu(B_r(x)) \leq \frac {1} {c_d{\cal H}_\infty^s(A)} r^s
\]
for all $r>0, x\in \R^d$. For $\delta\in (0,1)$ let $\nu_\delta=(1-\delta)\mu_\eta + \delta \nu$ and denote

\[
Z_\delta = \inf_{x\in E} \inf_{r>0} \frac {r^s} {\nu_\delta(B_r(x))}.
\]

Choose some $x\in E, r>0$. If $B_r(x)\cap A=\emptyset$, then from \eqref{eq:choicemu}
\begin{align*}
\nu_\delta(B_r(x))&= (1-\delta) \mu_\eta(B_r(x)) \\
&\leq \frac {1-\delta} {1-\eta} \cdot\frac {r^s} {\varphi^s(E)}=: G_1(\delta).
\end{align*}
Otherwise, by definitions of $A$ and $\nu$
\begin{align*}
\nu_\delta(B_r(x)) &= (1-\delta) \mu_\eta(B_r(x)) + \delta \nu(B_r(x)) \\
&\leq r^s\left(\frac {(1-\delta)(1-\varepsilon)} {\varphi^s(E)} + \frac \delta {c_d{\cal H}^s_\infty(A)}\right)=:G_2(\delta).
\end{align*}

For $\delta > \eta$ we have $G_1(\delta) < r^s/\varphi^s(E)$. If, contrary to the claim, we have
\begin{equation} \label{eqn:delta}
{\cal H}^s_\infty(A) > \varphi^s(E) \cdot (c_d(1-\varepsilon+\frac {\varepsilon} {\eta}))^{-1}
\end{equation}
then $G_2(\eta) < r^s/\varphi^s(E)$. It follows that for some $\delta>\eta$ 

\[
Z_\delta \geq \max(r^s/G_1(\delta), r^s/G_2(\delta)) > \varphi^s(E),
\]
which is a contradiction with the definition of $\varphi^s$. Thus, \eqref{eqn:delta} cannot hold.
\end{proof}

We can now compare $\varphi^s(E)$ to $\cal H^s_\infty(E)$. Let $(D_i)$ be the family of balls $B_r(x)$ with $x\in E$ for which $\mu_\eta(B_r(x)) \geq (1-\epsilon) r^s/Z$. This family can be infinite (even uncountable), but it contains a ball of maximal radius (possibly more than one). We will inductively construct a subfamily $(E_j)\subset (D_i)$ in the following way. We take the largest ball from $(D_i)$, which is the first ball in $(E_j)$. We then inductively add to $(E_j)$ the largest ball from $(D_i)$ not contained in $\bigcup_j 3E_j$ (where $3E_j$ means the ball with the same center as $E_j$ but three times larger radius). This way we construct a family of disjoint balls $B_r(x)$ satisfying $\mu_\eta(B_r(x)) \geq (1-\epsilon) r^s/Z$ and such that $\bigcup B_{3r}(x)\supset E\setminus A$.

From now on, consider the sequences $\varepsilon_n\to 0$ and $\eta_n=\varepsilon_n^2$ fixed. Then, by Lemma \ref{lem:Hslittle} above, 
\[
\cal H^s_\infty (A_{\eta_n, \epsilon_n})\le \frac{\varphi^s(E) }{ c_d(1-\varepsilon_n+\tfrac {\varepsilon_n} {\eta_n})}\le \frac{\varphi^s(E)}{c_d(1-\epsilon_n +\tfrac 1{\epsilon_n})}=:\ell_n. 
\]
By subadditivity of $\cal H^s_\infty$, for all $n$
\begin{align*}
\cal H^s_\infty(E)&\le \cal H^s_\infty(E\setminus A_{\eta_n, \epsilon_n}) +\cal H^s_\infty(A_{\eta_n, \epsilon_n})\leq \sum_{B_r(x)\in E_j} (6r)^s + \ell_n\\
&\leq 6^s Z/(1-\epsilon_n) \sum \mu_\eta(B_r(x))+\ell_n \leq 6^s Z/(1-\varepsilon_n)+\ell_n.
\end{align*}
As $\epsilon_n,\ell_n\to 0$ and $Z\le\varphi^s(E)$, this finishes the proof of the proposition. 
\end{proof}

We finish the section with a lemma showing that for open sets $E$ the supremum in \eqref{eq:phi} is attained over absolutely continuous measures. 

\begin{lemma} \label{lem:phis}
There exists $\kappa_1>0$ such that for every open bounded set $E\subset \R^d$ there exists an absolutely continuous probability measure $\eta$ of bounded density, such that the support of $\eta$ is a finite union of disjoint $d$-dimensional cubes contained in $E$ and

\[
\phi^s(E) \leq \kappa_1 \cdot \inf_{x\in E}\inf_{r>0}\frac{r^s}{\eta(E\cap B_r(x))}.
\]
\end{lemma}
\begin{proof}
Fix $\varepsilon >0$. Let $\mu_1$ be a probability measure supported on $E$ such that 

\[
\phi^s(E) \leq (1+\varepsilon) \cdot \inf_{x\in E}\inf_{r>0}\frac{r^s}{\mu_1(E\cap B_r(x))}.
\]

For $\delta >0$ let $E_\delta$ denote the points in $E$ lying at distance greater than $\delta$ from $\partial E$. We choose $\delta$ so small that $\mu_1(E_\delta) \geq 1-\varepsilon$ and define

\[
\mu_2 = \frac 1 {\mu_1(E_\delta)} \mu_1|_{E_\delta}.
\]
Note

\begin{equation}\label{eq:mu2}
\inf_{x\in E}\inf_{r>0}\frac{r^s}{\mu_2(E\cap B_r(x))} \geq (1-\varepsilon) \inf_{x\in E}\inf_{r>0}\frac{r^s}{\mu_1(E\cap B_r(x))}.
\end{equation}

Let $f_\delta$ be the normalized characteristic function of $B_\delta(0)$ and define
\[
d\mu_3(x) = \int f_\delta(x-y) d\mu_2(y).
\]
This is an absolutely continuous probability measure with density bounded by $(\lambda(B_\delta(0)))^{-1}$. For $x\in E$ and $r\geq\delta$ we have
\begin{equation}\label{eq:estimate1}
\mu_3(B_r(x)) \leq \mu_2(B_{r+\delta}(x)) \leq \mu_2(B_{2r}(x)),
\end{equation}
since the measure $\mu_3$ is obtained from the measure $\mu_2$ by redistributing it inside a $\delta$-neighbourhood of each point. For $x\in E$ and $r<\delta$ we have

\begin{equation}\label{eq:estimate2}
\mu_3(B_r(x)) \leq \frac {r^d} {\delta^d} \mu_2(B_{r+\delta}(x)) \leq \frac {r^d} {\delta^d} \mu_2(B_{2\delta}(x)).  
\end{equation}
Here we use the fact that for any $x$ the density of $\mu_3$ at $x$ equals $(\pi_d)^{-1}
\delta^{-d} \mu_2(B_\delta(x))$, where $\pi_d$ is the volume of a $d$-dimensional ball. By \eqref{eq:estimate1} and \eqref{eq:estimate2}, for every $x\in E$ and $r>0$ one can find $r'>0$ such that

\begin{equation}\label{eq:tocompare}
\frac {r^s} {\mu_3(B_r(x))} \geq 2^{-s} \frac {(r')^s} {\mu_2(B_{r'}(x))}.
\end{equation}

Finally, we choose some finite union $F\subset E$ of disjoint cubes such that $\mu_3(F) \geq 1-\varepsilon$ and define 

\[
\eta= \frac 1 {\mu_3(F)} \mu_3|_F.
\]
We have

\[
\inf_{x\in E}\inf_{r>0}\frac{r^s}{\eta(E\cap B_r(x))} \geq (1-\varepsilon) \inf_{x\in E}\inf_{r>0}\frac{r^s}{\mu_3(E\cap B_r(x))}. 
\]
Combining this with equations \eqref{eq:tocompare} and \eqref{eq:mu2}, and recalling the choice of $\mu_1$, we finish the proof with $\kappa_1$ arbitrarily close to $2^s$. 
\end{proof}

%Let us note an important property of the definition of $\varphi^s$. For any given $x\in E$ the infimum over $r$ is obtained for $r\leq \diam E$. Indeed, if we increase $r$ further we will only increase $r^s$ while and $\mu(E\cap B_r(x))$ will stay unchanged. We can thus take in the definition of $\phi^s$ infimum over $r<\diam E$ instead of infimum over all positive $r$. Hence, when $\diam E<1$ we can say that $\varphi^s(E)\geq K$ implies

%\begin{equation} \label{eqn:varfi}
%s \leq \sup_\mu \inf_{x\in E}\inf_{r\in (0,1)} \frac {\log \left(\mu(E\cap B_r(x)) \cdot K \right)} {\log r}.
%\end{equation}

\section{Statement of results}\label{sec:theorems}
The following is the main theorem of this article. 
\begin{theorem} \label{thm:main}
Let $(B_i)$ be a sequence of balls in $[0,1]^d\subset \R^d$ such that $\lambda(\limsup_{i\to \infty}B_i)=1$. Let $(E_i)$ be a sequence of open sets, such that $E_i\subset B_i$.
Define
\[
s=\sup \{t\mid \lambda(\limsup \{B_i\mid \varphi^t(E_i)\ge \lambda(B_i)\})=1\}.
\]
Then
\[
\dimH \limsup E_i\ge s
\]
and
\[
\dimP \limsup E_i= d.
\]
\end{theorem}

The claim $\dimP \limsup E_i= d$ also follows by observing that $\limsup E_i$ is a dense $G_\delta$-set \cite[Fact 12]{SV}, but we give a direct proof. We will actually prove the following result; it is clear that Theorem \ref{thm:main} is an immediate corollary.

\begin{theorem} \label{thm:aux}
Let $(B_i)$ be a sequence of balls in $[0,1]^d\subset \R^d$ such that $\lambda(\limsup_{i\to \infty}B_i)=1$. Let $(E_i)$ be a sequence of open sets, such that $E_i\subset B_i$. Assume that for some $s\geq 0$ each pair $(B_i, E_i)$ satisfies
\[
\phi^s(E_i) \geq \lambda(B_i).
\]
Then
\[
\dimH \limsup E_i \geq s
\]
and
\[
\dimP \limsup E_i= d.
\]
\end{theorem}

\begin{remark}
In particular, the sets $(E_i)$ being balls as in \cite{BV} or ellipsoids as in \cite{WWX} satisfy the assumptions of Theorem \ref{thm:main}, so that Theorem \ref{thm:main} recovers these dimension results. Furthermore, as is the case in \cite{BV, WWX}, the lower bound we provide can be sharp, see e.g. \cite[Corollary 5.1]{WWX}. 
\end{remark}

By Proposition \ref{prop:content}, we have the following corollary:

\begin{corollary} \label{cor:cusps}
Let $(B_i)$ be a family of balls in $[0,1]^d$, such that $\lambda (\limsup B_i) =1$. For some $s\in (0,d)$ for every $i$ let $E_i\subset B_i$ be an open set satisfying $\cal H^s_\infty(E_i) \geq \lambda(B_i)$. Then $\dimH \limsup E_i \geq s$ and $\dimP \limsup E_i =d$.
\end{corollary}

\begin{example} The following example shows that the assumption that $\limsup B_i$ have full Lebesgue measure cannot be relaxed to positive Lebesgue measure. 

Denote by $\Sigma_*$ the countable set $\bigcup_{n=0}^\infty \{0,1\}^n$. For a word $\omega\in \Sigma_*$ let $|\omega|$ denote its length. We will construct a countable family of closed intervals $B_\omega\in [0,1], \omega\in\Sigma_*$ such that $\lambda(\limsup B_\omega)>0$ but for every $a>1$ $\limsup E_\omega=\emptyset$, where $E_\omega$ is an interval with the same center as $B_\omega$ but with diameter $|B_\omega|^a$. In particular, there is no nontrivial lower bound for the dimension of $\limsup E_\omega$

Let $a_n=1/2(n+1)^2$. Let $B_\emptyset=[0,1]$. Inductively, for $\omega\in \Sigma_*$, define $J_\omega=a_{|\omega|} B_\omega^o$. Then let$B_{\omega 0}$ and $B_{\omega 1}$ be the left and right components of $B_\omega\setminus J_\omega$.

For every $n\geq 0$ we have $\lambda (\bigcup_{\omega \in \{0,1\}^n} B_\omega) = \prod_{i=0}^{n-1} (1-a_i) > \kappa := \prod_{i=0}^\infty (1-a_i) >0$. In particular, $|B_\omega| \geq \kappa \cdot 2^{-|\omega|}$ and $\lambda (\limsup B_\omega) = \kappa >0$ as desired. 

Choose $a>1$ and define $E_\omega$. There is $N=N(a)$ such that for all $n>N$ we have $(\kappa 2^{-n})^a < a_n \kappa 2^{-n}$. Thus, for $|\omega|>N$ we have $E_\omega \subset J_\omega$, hence $E_\omega$ eventually become disjoint with all $E_\nu, |\nu|>|\omega|$. This implies $\limsup E_\omega = \emptyset$.

\end{example}

The strategy of the proof of Theorem \ref{thm:main} is as follows: We will construct a large Cantor subset $F$ of $\limsup E_i$, define a mass distribution $\mu$ on the construction tree of $F$ and estimate the local dimension of $\mu$. This will give a lower bound to the dimension.

\section{Construction of the Cantor subset}\label{sec:construction}

We note that we can freely assume that the size of balls $B_i$ forms a nonincreasing sequence converging to 0. Indeed, the statement of the theorem does not depend on the order of $B_i$'s, and moreover if the size of the balls $B_i$ has a non-zero lower bound and if $\varphi^s(E_{n_i})>\lambda(B_{n_i})$ for some $s>0$ and some subsequence $E_{n_i}$ then by the definition of $\varphi$ we will have a nonzero lower bound for $\lambda(E_{n_i})$, and hence for $\lambda(\limsup E_{n_i})$ as well.

For a ball $B$, denote by $MB$ a ball of the same center and $M$ times the radius. The following lemma has been proven as \cite[Lemma 5]{BV}, but, as it is a crucial ingredient in the construction of the Cantor set $F$, for completeness we present a proof.
\begin{lemma}\label{lem:kgb}
Assume $\lambda(\limsup B_i)=1$. Then there exists $\kappa_2>0$ such that for every cube $C\subset [0,1]^d$ there exists a finite family of balls $B_{n_i}\subset C$ such that the balls $3B_{n_i}$ are pairwise disjoint and that
\[
\sum \lambda(B_{n_i}) \geq \kappa_2 \lambda(C).
\]
\end{lemma}
\begin{proof}
Let $r$ denote the side of $C$. As the diameter of balls $B_i$ converges to 0, for any positive $\varepsilon$ we know that up to a set of zero Lebesgue measure

\[
\bigcup_{i; B_i\subset C} 3B_i \supset \bigcup_{i; B_i\subset C} B_i \supset (1-\varepsilon)C,
\]
where $(1-\varepsilon)C$ denotes a cube of the same center as $C$ but of side $(1-\varepsilon)r$.

Applying the $5r$-covering theorem \cite[Theorem 2.1]{M}, we find a (finite or countable) subfamily of balls $B_{i_k}\subset C$ such that up to a set of measure zero

\[
\bigcup 15 B_{i_k} \supset (1-\varepsilon)C,
\]
and that the balls $3B_{i_k}$ are disjoint. Hence,

\[
\sum \lambda(B_{i_k}) \geq r^d (1-\varepsilon)^d 15^{-d}
\]
and we can choose a finite subfamily such that

\[
\sum \lambda(B_{i_k}) \geq r^d (1-2\varepsilon)^d 15^{-d}.
\]
As $\lambda(C)=r^d$, we are done.
\end{proof}

%If necessary, reorganise the sequence $(B_i)$ make $(\diam B_i)$ decreasing.

We will now begin the construction of the Cantor set $F$. First, for every set $E_i$ denote by $\eta_i$ the absolutely continuous measure provided by Lemma \ref{lem:phis} and by $\ell_i$ the supremum of its density. We will denote by ${\tilde E}_i$ the support of $\eta_i$, which by Lemma \ref{lem:phis} is a finite union of disjoint cubes. It suffices to give the lower bound for $\dimH \limsup {\tilde E}_i$ and $\dimP \limsup {\tilde E}_i$.

We will now inductively construct a family of sets $F_0\supset F_1\supset\ldots$ such that each $F_j; j\geq 1$ is a finite union of some ${\tilde E}_i$'s. Clearly,
\[
F := \bigcap F_j \subset \limsup {\tilde E}_i \subset \limsup E_i.
\]
In the next section we will proceed to distribute a measure $\mu$ on $F$.

Start with the cube $F_0=[0,1]^d$. Applying Lemma \ref{lem:kgb} to the cube $F_0$ we can find a finite family of disjoint balls ${\cal F}_1 \subset \{B_i\}$ such that
\[
\sum_{B_i\in {\cal F}_1} \lambda(B_i) > \kappa_2,
\]
where $\kappa_2$ is from Lemma \ref{lem:kgb}. Let $F_1 = \bigcup_{B_i\in {\cal F}_1} {\tilde E}_i$.

Fix some sequence $\epsilon_j\searrow 0$ and recall that each ${\tilde E}_i$ for $B_i\in \cal F_1$ is a union of cubes. Denote by $r_1$ the diameter of the smallest of these cubes. Further, let
\[
\tilde{r}_1 = \min(r_1, (\kappa_2\cdot \min\{\frac {1} {\ell_i \lambda(B_i)}; B_i\in {\cal F}_1\})^{1/\epsilon_1}).
\]
Now divide all the components of all ${\tilde E}_i$ with $B_i\in {\cal F}_1$ into cubes $D_1^{(1)},\dots, D_{N_1}^{(1)}$ of diameter between $\tilde{r}_1/2$ and $\tilde{r}_1$ (notice that different components might need to be divided into cubes of different size). These cubes will be where the construction continues.

We carry on inductively. Assume that the notions $\cal F_{j-1}$, $\tilde r_{j-1}$, 
\[
F_{j-1}=\bigcup_{B_i\in \cal F_{j-1}}{\tilde E}_i
\]
and the cubes $D_1^{(j-1)}, \dots, D_{N_{j-1}}^{(j-1)}\subset F_{j-1}$ of diameter between $\tilde{r}_{j-1}/2$ and $\tilde{r}_{j-1}$ as above have been defined.

Now apply Lemma \ref{lem:kgb} to each $D_k^{(j-1)}$, $k=1, \dots, N_{j-1}$. Obtain in this way a family ${\cal F}_j$ of balls $B_i$ such that for $D^{(j-1)}_k$, $k=1, \dots, N_{j-1}$, 
\[
\sum_{B_i\in {\cal F}_j; B_i\subset D_k^{(j-1)}} \lambda(B_i) \geq \kappa_2 \lambda(D_k^{(j-1)}).
\]
Let $F_j = \bigcup_{B_i\in {\cal F}_j} {\tilde E}_i$.

Finally, define $r_j$ as the smallest diameter of cube components of $F_j$. Set
\[
\tilde{r}_j = \min(r_j, \tilde r_{j-1}\cdot (\kappa_2\cdot \min\{\frac {1} {\ell_i \lambda(B_i)}; B_i\in {\cal F}_j\})^{1/\epsilon_j}),
\]
and subdivide $F_j$ into cubes of diameter between $\tilde r_j/2$ and $\tilde r_j$ as above to continue.

\section{Construction of the mass distribution}

We will now construct a mass distribution on $F$, and proceed in the next section by investigating its local dimension. Begin by setting the notations
\[
\mathcal F_j(E)=\{B_i \in \mathcal F_{j}\mid B_i \subset E\}\textrm{ and }F_j(E)=\bigcup_{B_i \in \mathcal F_j(E)}{\tilde E}_i
\]
for $E\subset F_0$.

%Denote $\rho_i=\ell_i \lambda(B_i)$ for all $i$. 
We start with $\mu_0$ defined as the Lebesgue measure $\lambda$ restricted to $F_0$. As an intermediate step in the definition of $\mu_1$, in the first level of construction $F_1$, define
\[
\nu_1|_{B_i}=\frac{\mu_0|_{B_i}}{\sum_{B_k \in \mathcal F_1}\mu_{0} (B_k)}
\]
for all $i$ such that $B_i\in \mathcal F_1$, and no mass elsewhere. Then define, for $B_i\in \mathcal F_1$ and ${\tilde E}_i\subset B_i$, the measure $\mu_1$ supported on $F_1$ by setting
\[
\mu_1|_{{\tilde E}_i}=\nu_1(B_i) \cdot \eta_i. 
\]
Continue in this way; assume that $\mu_{n-1}$ has been defined on the sets ${\tilde E}_i$ with $B_i\in \mathcal F_{n-1}$. Let $B_k\in \mathcal F_n$, $B_k\subset D_j^{(n-1)}$, where $D_j^{(n-1)}$ is a cube of side length approximately $\tilde r_{n-1}$ from the cube decomposition of ${\tilde E}_i$. Then define
\[
\nu_n|_{B_k}=\frac{\mu_{n-1}(D_j^{(n-1)}) \lambda|_{B_k}}{\sum_{B_\ell \in \mathcal F_n(E_i); B_\ell\subset D_j^{(n-1)}}\lambda(B_\ell)},
\]
and for each ${\tilde E}_k\subset B_k\in \mathcal F_n$
\[
\mu_n|_{{\tilde E}_k}=\nu_n(B_k)\cdot \eta_k,
\]
obtaining a measure supported on $F_n$. 

Notice that $(\mu_n)$ is a sequence of probability measures supported on the compact set $[0,1]^d$, so that it has a weakly convergent subsequence. Denote the limit of this subsequence by $\mu$, and notice that it is by construction supported on $F$. In fact, $\mu_n(B_i)=\mu_{n+k}(B_i)$ for all $k\ge 0$, for all $B_i\in \mathcal F_n$, and similarly for ${\tilde E}_i\subset B_i\in \mathcal F_n$.

\section{Estimating the local dimension}\label{sec:localdim}

We now bound the local dimension of $\mu$. Pick a point $x\in F$ and $r>0$. We want to give an estimate to the $\mu$-measure of the ball $B_r(x)$. Let $n$ be such that $\tilde r_{n}<r\le \tilde r_{n-1}$. Since $x\in F$, we can write $x\in B_{i_n}\subset B_{i_{n-1}} \subset \ldots\subset B_{i_1}$, with $B_{i_k}\in {\cal F}_k$ for all $k$.

%Then $x\in E_{i_0}$ for some $E_{i_0}$ such that $B_{i_0}\in \mathcal F_{n+1}$, and in fact, we find $E_{i_1}\subset \dots \subset E_{i_{n}}$ such that $x\in E_{i_j}$, and $B_{i_j}\in \mathcal F_{n+1-j}$.

%Focus first on the measure $\mu_{n+1}$. Recall from the construction of $F$ that on each level we divide the construction sets $E_i$ into cubes of approximately equal size. In particular, each $E_{i_1}$ with $B_{i_1} \in \mathcal F_{n}$ is divided into cubes of diameter approximately $\tilde r_{n+1}$. Let $Q$ be a cube in this decomposition. Then
%\[
%\mu_{n}(Q)\le (1/\kappa)^{n} \frac{\lambda(B_{i_n})\dots \lambda(B_{i_1})}{\lambda(E_{i_n})\dots \lambda(E_{i_1})}\lambda(Q).
%\]
%We will use the notation
%\begin{align*}
%C_n(B_{i_2}):=\log \left((1/\kappa)^{n} \frac{\lambda(B_{i_n})\dots \lambda(B_{i_2})}{\lambda(E_{i_n})\dots \lambda(E_{i_2})}\right)&= \log \frac 1 \kappa\log (\frac{\lambda(B_{i_n})}{\kappa \lambda(E_{i_n})})+\dots +\log (\frac{\lambda(B_{i_2})}{\kappa \lambda(E_{i_2})}),
%\end{align*}
%so that
%\[
%\mu_{n}|_{B_{i_0}}\leq \frac {\lambda(B_{i_1})} {\lambda(E_{i_1})} \cdot \exp(C_n(B_{i_2}))\cdot\lambda|_{B_{i_0}}.
%\]
%Further, recallling the definition of $\tilde r_n$,
%\begin{align*}
%\tilde r_n&\le1/\epsilon_1\log (\frac{\lambda(B_{i_n})}{\kappa\lambda(E_{i_n})})+\log \tilde r_{n-1}\\
%&=1/\epsilon_1\log (\frac{\lambda(B_{i_n})}{\kappa\lambda(E_{i_n})})+\dots +1/\epsilon_n\log (\frac{\lambda(B_{i_1})}{\kappa\lambda(E_{i_1})}).
%\end{align*}

There are two cases to consider: $\diam B_{i_n} \leq r < \tilde{r}_{n-1}$ and $\tilde{r}_{n} \leq r < \diam B_{i_n}$.

\subsection*{Case 1: $\diam B_{i_n} \leq r < \tilde{r}_{n-1}$}

Recall that in the construction we divide the set ${\tilde E}_{i_{n-1}}$ into the $(n-1)$-st generation cubes $D_j^{(n-1)}, j=1, \dots, N_{n-1}$ of diameter approximately $\tilde{r}_{n-1}$, and for each of them 
\[
\mu(D_j^{(n-1)})=\mu_{n-1}(D_j^{(n-1)}). 
\]
Further, 
\[
\mu_{n-1}(D_j^{(n-1)})=\nu_{n-1}(D_j^{(n-1)})\eta_{i_{n-1}}(D_j^{(n-1)})\le \nu_{n-1}(D_j^{(n-1)})\ell_{i_{n-1}}\lambda(D_j^{(n-1)}). 
\]
To continue, set the notation
\[
D_j^{(n-1)}\subset\tilde  E_{i_{n-1}}\subset B_{i_{n-1}}\subset D_k^{(n-2)}\subset \tilde E_{i_{n-2}}\subset B_{i_{n-2}}. 
\]
Then by the definition of $\nu_{n-1}$,
\[
\nu_{n-1}(D_j^{(n-1)})=\frac{\mu_{n-1}(D_k^{(n-2)})}{\sum_{B_\ell\in \cal F_{n-1}; B_\ell \subset D_k^{(n-1)}}\lambda(B_\ell)}\cdot\lambda(B_{i_{n-1}}),
\]
where $\cal F_{n-1}$ was chosen using Lemma \ref{lem:kgb} so that 
\[
\sum_{B_\ell\in \cal F_{n-1}; B_\ell \subset D_k^{(n-2)}}\lambda(B_\ell)\ge \kappa_2\lambda(D_k^{(n-2)}). 
\]
Combining the above, we obtain 
\[
\mu_{n-1}(D_j^{(n-1)})\le \frac{\mu_{n-1}(D_k^{(n-2)})}{\lambda(D_k^{(n-2)})}\cdot\frac{\ell_{i_{n-1}}\lambda(B_{i_{n-1}})}{\kappa_2}\cdot\lambda(D_j^{(n-1)}). 
\]
Using this inductively, we end up with 
\begin{equation}\label{eq:muonD}
\mu(D_j^{(n-1)}) = C_{n-1}(D_j^{(n-1)}) \cdot \lambda(D_j^{(n-1)}),
\end{equation}
where 
\begin{equation}\label{eq:Cn}
C_{n-1}(D_j^{(n-1)}) \leq \frac {\lambda(B_{i_1})\cdots \lambda(B_{i_{n-1}}) \cdot \ell_{i_1}\cdots \ell_{i_{n-1}}} {\kappa_2^{n-1}}.
\end{equation}
Let now $D^{(n-1)}$ be the $(n-1)$-st generation cube containing $x$.
We will write $C_{n-1}(x)$ for $C_{n-1}(D^{(n-1)})$.

Recall that $\tilde r_{n-1}$ was chosen in such a way that
\begin{align*}
|\log\tilde r_{n-1}|&\le1/\epsilon_{n-1}\log (\frac{\ell_{i_{n-1}}\lambda(B_{i_{n-1}})}{\kappa_2})+\log \tilde r_{n-2}\\
&=1/\epsilon_1\log (\frac{\ell_{i_1}\lambda(B_{i_1})}{\kappa_2})+\dots +1/\epsilon_{n-1}\log (\frac{\ell_{i_{n-1}}\lambda(B_{i_{n-1}})}{\kappa_2}).
\end{align*}
In particular, by \eqref{eq:Cn} and the choice of $1/\epsilon_n\to \infty$, we have
\begin{equation} \label{eqn:cn}
\lim_{n\to\infty} \max_y \frac {|\log C_n(y)|} {|\log \tilde r_n|} =0.
\end{equation}

Using, essentially, \eqref{eq:muonD}, in the cube $D^{(n-1)}$ we find a collection of balls $B_i\in {\cal F}_n$ such that each of them satisfies
\begin{equation}\label{eq:littleballs}
\mu(B_i)=\nu_n(B_i) \leq \frac 1 \kappa_2 \cdot C_{n-1}(x) \lambda(B_i)
\end{equation}
and $3B_i$ are disjoint. Observe that, since $r<\tilde r_{n-1}$, the ball $B_r(x)$ can intersect at most $5^d$ of the $(n-1)$-st generation cubes. Furthermore, if indeed there is some $D_j^{(n-1)}$ such that $y\in D_j^{(n-1)}\cap F\cap B_r(x)$ then $B_r(x)\cap D_j^{(n-1)} \subset B_{2r}(y)\cap D_j^{(n-1)}$. Hence, $\mu(B_r(x))\le 10^d\mu(B_r(x)\cap D^{(n-1)})$, and we continue with the latter. We have
\[
\mu(B_r(x)\cap D^{(n-1)}) \leq \sum_{B_i\in {\cal F}_n; B_i\cap B_r(x)\neq\emptyset} \mu(B_i).
\]
However, by the construction, balls $3B_{i_n}$ and $3B_i$ are disjoint for any $i\neq i_n$, and in particular, since $x\in B_{i_n}$, we have $x\notin 3B_i$. Hence, if $B_r(x)$ intersects $B_i$ then $\diam B_i \leq r$, and we have from \eqref{eq:littleballs} and the disjointness of $B_i$
\[
\mu(B_r(x)\cap D^{(n-1)}) \leq \lambda(B_{2r}(x)) \cdot \frac {C_{n-1}(x)}{\kappa_2}.
\]

Summing up the estimates from above, we get
\[
\mu(B_r(x)) \leq 20^d \cdot \frac 1 \kappa_2 \max_y C_{n-1}(y)r^d
\]
and hence by \eqref{eqn:cn}, for $\diam B_{i_n} \leq r < \tilde{r}_{n-1}$ we have
\begin{equation}\label{eqn:case1}
\frac {\log \mu(B_r(x))} {\log r} \geq d - q_n
\end{equation}
with $q_n\to 0$.

\subsection*{Case 2: $\tilde{r}_{n} \leq r < \diam B_{i_n}$}

In this case $B_r(x)$ is not going to intersect any $B_i\in {\cal F}_n, i\neq i_n$. Hence, $\mu(B_r(x)) = \mu(B_r(x) \cap {\tilde E}_{i_n})$.

Consider the distribution of measure $\mu$ on ${\tilde E}_{i_n}$. We have
\begin{equation} \label{eqn:cn2}
\mu_n|_{{\tilde E}_{i_n}} \leq C_{n-1}(x) \kappa_2^{-1} \lambda(B_{i_n}) \cdot \eta_{i_n}.
\end{equation}
%and
%\[
%C_n(x) \leq C_{n-1}(x) \kappa_2^{-1} \ell_{i_n} \lambda({\tilde E}_{i_n}).
%\]

Hence, for each of the $n$-th level cubes $D_j^{(n)}$ we have

\[
\mu(D_j^{(n)})=\mu_n(D_j^{(n)}) \leq C_{n-1}(x) \kappa_2^{-1} \lambda(B_{i_n}) \cdot \eta_{i_n}(D_j^{(n)}).
\]
We note that these are $n$-th generation cubes, of size approximately $\tilde r_n$, not the $(n-1)$-st generation cubes we considered in the previous case. However, we do not yet know how exactly $\mu$ is distributed on each $D_j^{(n)}$ -- this will be decided on the following stages of the construction. Nevertheless, we can write

\[
\mu(B_r(x)\cap {\tilde E}_{i_n}) \leq \sum_{D_j^{(n)}; D_j^{(n)}\cap B_r(x) \neq \emptyset} \mu(D_j^{(n)})
\]
and we also know that if $D_j^{(n)}\cap B_r(x) \neq \emptyset$ then $D_j^{(n)}\subset B_{r+\tilde r_n}(x)$. Combining this with \eqref{eqn:cn2} we get

\[
\mu(B_r(x)) \leq C_{n-1}(x) \kappa_2^{-1} \lambda(B_{i_n}) \eta_{i_n}(B_{r+\tilde r_n}(x)\cap {\tilde E}_{i_n}).
\]

Note that $r+\tilde r_n \leq 2r$. By the definition of $\eta_{i_n}$ and the assumption $\lambda(B_i)\ge \phi^s(E_i)$ we have

\[
\eta_{i_n}(B_{2r}(x)) \leq \frac {(2r)^s \kappa_1} {\lambda(B_{i_n})},
\] 
and, using \eqref{eqn:cn}
\begin{equation} \label{eqn:case2}
%\begin{aligned} 
%&\frac {\log \mu(B_r(x))} {\log r} \\
%&\qquad\quad\geq \frac {\log C_{n-1}(x) + \log \kappa_1 -\log \kappa_2} {\log r} + \frac {\log \eta_{i_n}(B_r(x)\cap E_{i_n}) + \log \lambda(B_{i_n}) - \log \kappa_1} {\log r -\log 2}\\
%&\qquad\quad\geq s - q_n
%\end{aligned}
\frac {\log \mu(B_r(x))} {\log r} \geq \frac {s\log 2 +\log C_{n-1}(x) + \log \kappa_1 -\log \kappa_2} {\log r} + s \geq s+q_n
\end{equation}
with $q_n \to 0$.

%\section{Conclusion of the proof}

We finish the proof of Theorem \ref{thm:aux} applying the mass distribution principle \cite[Proposition 2.3]{F} to \eqref{eqn:case1} and \eqref{eqn:case2}.

%\newpage
\bibliography{ref}

\begin{thebibliography}{WW}

%\addcontentsline{toc}{section}{Bibliografia}

\thispagestyle{empty}
%\markright{Bibliografia}

%\begin{enumerate}

\bibitem[AB]{AB}
D. Allen, S. Baker,
\newblock A General Mass Transference Principle,
\newblock preprint 2018, arXiv:1803.02654.

\bibitem[AT]{AT}
D. Allen, S. Troscheit,
\newblock The Mass Transference Principle: Ten years on,
\newblock preprint 2016, arXiv:1704.06628.


\bibitem[BV]{BV}
V. Beresnevich, S. Velani,
\newblock A mass transference principle and the Duffin-Schaeffer conjecture for Hausdorff measures.
\newblock {\it Ann. of Math.} 164 (2006), 971--992.

\bibitem[C]{C}
L. Carleson,
\newblock {\it Selected Problems on Exceptional Sets},
\newblock Van Nostrand 1967.

\bibitem[F]{F}
K. J. Falconer,
\newblock {\it Techniques in Fractal Geometry},
\newblock {John Wiley and Sons, Chichester,} 1997.

\bibitem[F2]{F2}
K. J. Falconer,
\newblock The Hausdorff dimension of self-affine fractals,
\newblock {\it  Math. Proc. Cambridge Philos. Soc.} 103 (1988), 339--350.

%\bibitem[F3]{F3}
%H. Federer, 
%Geometric Measure Theory, {\it Springer-Verlag, New York,} 1969. 

\bibitem[FJJS]{FJJS}
D.-J. Feng, E. J\"arvenp\"a\"a, M. J\"arvenp\"a\"a, V. Suomala,
\newblock Dimensions of random covering sets in Riemann manifolds,
\newblock {\it Ann. of Prob.} 46 (2018), 1542--1596.  

\bibitem[J]{J}
S. Jaffard,
\newblock  Construction of functions with prescribed H\"older and chirp exponents,
\newblock {\it Revista Matem\'atica Iberoamericana} 16:2 (2000), 331--349.

\bibitem[M]{M}
P. Mattila,
\newblock {\it Geometry of Sets and Measures in Euclidean Spaces},
\newblock {Cambridge University Press, Cambridge,} 1995.

\bibitem[P]{P}
T. Persson
\newblock A note on random coverings of tori,
\newblock {\it Bull. Lond. Math. Soc.} 47 (2015), 7--12.

\bibitem[SV]{SV}
K. Simon, L. V\'ag\'o, 
\newblock Singularity versus exact overlaps for self-similar measures, 
\newblock{\it Proc. Amer. Math. Soc.} 147 (2019), 1971--1986.

\bibitem[WWX]{WWX}
B.-W. Wang, J. Wu, J. Xu,
\newblock Mass transference principle for limsup sets generated by rectangles,
\newblock {\it  Math. Proc. Cambridge Philos. Soc.} 158 (2015), 419--437. 

%\end{enumerate}
\end{thebibliography}

\end{document}